\documentclass{siamltex1213}
\usepackage[utf8]{inputenc} 
\usepackage[english]{babel}
\usepackage[T1]{fontenc}
\usepackage{epstopdf}
\usepackage[toc,page]{appendix}
\usepackage{caption}
\usepackage{subcaption}
\usepackage{amsmath}
\usepackage{graphicx}
\usepackage{epsfig}
\usepackage{amsfonts}
\usepackage{amssymb}
\usepackage{eufrak}
\usepackage[toc,page]{appendix}
\title{Least Squares Shadowing method for sensitivity analysis of differential equations}

\author{Mario CHATER\thanks{Department of Aeronautics and Astronautics, MIT, 77 Mass Ave, Cambridge, MA 02139, USA,  email addresses: \email{mario\_chater@hotmail.com} (Mario Chater), \email{niangxiu@mit.edu} (Angxiu Ni), \email{blonigan@mit.edu} (Patrick J. Blonigan), \email{qiqi@mit.edu} (Qiqi Wang)} \and Angxiu NI\footnotemark[1] \and Patrick J. Blonigan\footnotemark[1] \and Qiqi WANG\footnotemark[1]}

\begin{document}
\maketitle

\slugger{mms}{xxxx}{xx}{x}{x--x}

\begin{abstract}
For a parameterized hyperbolic system $\frac{du}{dt}=f(u,s)$ the derivative of the ergodic average $\langle J \rangle = \lim_{T \to \infty}\frac{1}{T}\int_0^T J(u(t),s)$ to the parameter $s$ can be computed via the Least Squares Shadowing algorithm (LSS). We assume that the sytem is ergodic which means that $\langle J \rangle$ depends only on $s$ (not on the initial condition of the hyperbolic system). The algorithm solves a constrained least squares problem and, from the solution to this problem, computes the desired derivative $\frac{d\langle J \rangle}{ds}$. The purpose of this paper is to prove that the value given by the LSS algorithm approaches the exact derivative when the timespan used to formulate the least squares problem grows to infinity. It then illustrates the convergence result through a numerical example.  
\end{abstract}



\pagestyle{myheadings}
\thispagestyle{plain}

\section{Introduction}

Consider the differential equation parameterized by $s \in \mathbb{R}$ and governing $u(t)\in U$ where $U$ is a Hilbert space :
\begin{align}
\label{difeq}
\left \{ \begin{array}{ll}
\frac{du}{dt}=f(u,s)\\
u(0)=u_0 & u_0 \in U
\end{array}\right.
\end{align} 
The differential equation is assumed to be uniformly hyperbolic (details in section \ref{sec:uniformhyperbolicity}). We are also given a $C^1$ cost function $J(u,s): U \times \mathbb{R} \to \mathbb{R}$ and assume that the system is \textit{ergodic}, i.e., the infinite time average :
\begin{equation}
\label{ergodicity}
\big\langle J\big\rangle(s)=\lim_{T \rightarrow +\infty} \frac{1}{T}\int_0^T J(u(t),s) dt
\end{equation}
depends on $s$ but does not depend on the initial condition $u(0)$. The differentiability of $\langle J \rangle $ with respect to $s$ has been proven by Ruelle \cite{differentiability}. Obtaining an estimation of $\frac{d\langle J \rangle}{ds}$ is crucial in many computational and engineering problems. Indeed, many applications involve simulations of nonlinear dynamical systems that exhibit a chaotic behavior. For instance, chaos can be encountered in the following fields : climate and weather prediction \cite{weather}, turbulent combustion simulation \cite{combustion}, nuclear reactor physics \cite{nuclear}, plasma dynamics in fusion \cite{plasma} and multi-body problems \cite{nbody}. The quantities of interest are often time averages or expected values of some cost function $J$. Estimating the derivative of $\langle J \rangle$ is particularly valuable in:\\
\begin{itemize}
\item \textbf{Numerical optimization}. The derivative of $\langle J \rangle$ with respect to a design parameter $s$ is used by gradient-based algorithms in order to efficiently optimize the design parameters in high dimensional design spaces (see \cite{designopti}).  
\item \textbf{Uncertainty quantification}. The derivative of $\langle J \rangle$ with respect to a parameter $s$ gives a useful information for assessing the error and uncertainty in the computed $\langle J \rangle$ (see \cite{uncertainty}).
\end{itemize}
For example, we could obtain a useful information about the impact of mankind on the climate by computing the derivative of the long time averaged global mean temperature with respect to the amount of anthropogenic emissions (\cite{emission} shows how sensitivity analysis is used in climate studies). In the simulation of a turbulent airflow over an aircraft, estimating the derivative of the long time averaged drag to a shape design parameter is of extreme importance for engineers allowing them to improve their design \cite{airfoil}. It has been shown that in many of these practical examples, the quantities of interest exhibit ergodic properties, popularly known as \textit{chaotic hypothesis} \cite{chaotichypothesis}, \cite{chaotichypothesis2}. As opposed to Kalman filter and Bred vector approaches, we do not aim to infer the state of the system at any particular time. We perform sensitivity analysis only with respect to the parameters of the system since our objective function  only depends on statistics (long-time averages) of the dynamical system.\\

When it comes to computing $\frac{d\langle J \rangle}{ds}$, conventional methods based on linearizing the initial value problem (\ref{difeq}) become ill-conditioned when the system is chaotic. They compute derivatives that are orders of magnitude too large and the error grows exponentially larger as the simulation runs longer \cite{butterfly},\cite{explosion}. This failure is due to the so-called \textit{butterfly effect} and the explanation has been published by Lea et al. \cite{butterfly}.\\
Some algorithms have been developed to overcome this failure. Lea et al. proposed the ensemble adjoint method which applies the adjoint method to many random trajectories, then averages the computed derivatives \cite{butterfly}, \cite{ensembleadjoint}. However, the algorithm is computationally expensive even for small dynamical system such as Lorenz's one. Based on the fluctuation dissipation theorem, Abramov and Majda provided an algorithm that successfully computes the desired derivative \cite{abramov}. Nonetheless, this algorithm assumes the dynamical system to have an equilibrium distribution similar to the Gaussian distribution, an assumption often violated in very dissipative systems. Recent work by Cooper and Haynes has alliviated this limitation by introducing a nonparametric method to estimate the equilibrium distribution \cite{cooper}. More methods have been developed to compute $\frac{d\langle J \rangle}{ds}$, in particular the \textit{Least Squares Shadowing (LSS)} algorithm which computes it by solving a constrained least squares problem \cite{explosion}. The big advantage of this method is its simplicity since the least squares problem can easily be formulated and efficiently solved as a linear system. Compared to the previously presented methods, LSS is less sensitive to the dimension of the dynamical system and doesn't require any explicit knowledge about its steady-state distribution in phase space.\\

This paper provides a theoretical foundation for LSS by proving that it gives a useful estimation of $\frac{d\langle J \rangle}{ds}$ when the dynamical system is a uniformly hyperbolic flow. Compared to the discrete case (uniformly hyperbolic map) for which we already have a proof of convergence \cite{proof}, the continuous case is more difficult to deal with due to the apparition of the \textit{neutral subspace} (details in section \ref{sec:uniformhyperbolicity}). However, it is very important to treat the continuous case since most applications and real-life problems require a continuous description of the physics and involve differential equations.\\

In the next section, the mathematical formulation of convergence is introduced as well as theorem LSS which will be proved in the following sections. Section \ref{sec:uniformhyperbolicity} presents the concept of uniform hyperbolicity for the readers who are not familiar with the subject. Section \ref{sec4} points out the new behaviour and properties that come with continuous dynamical systems (as opposed to discrete maps). Section \ref{sec5} defines the shadowing direction and proves its existence as well as uniqueness. Section \ref{sec6} shows that the derivative of $\langle J \rangle$ can be computed using the shadowing direction and bounds the upper error. Section \ref{sec7} then demonstrates that the least squares problem gives a good approximation of the shadowing direction. Then, section \ref{sec8} uses all the previous results and  concludes the proof of theorem LSS by showing that the estimation error vanishes as the least squares problem increases in size. Finally, section \ref{sec9} presents a numerical example and illustrates the convergence result.

\section{LSS convergence theorem}
\label{sec2}

We begin by presenting the convergence result for the \textit{Least Squares Shadowing} method. For a trajectory $\{u(t)\}_{t\in(0,T)}$ satisfying the differential equation (\ref{difeq}), \textit{LSS} attempts to compute the derivative $\frac{d<J>}{ds}$ via
\begin{theorem}[THEOREM LSS]
Under ergodicity and hyperbolicity assumptions,
\begin{align*}
\frac{d \langle J \rangle}{ds}(s)&=\lim_{T \to \infty} \int_0^T \Big[(DJ(u(t),s))v^{\{T\}}(t)+\partial_sJ(u(t),s)+\eta^{\{T\}}(t)(J(u(t),s)-\langle J\rangle(s))\Big]dt
\end{align*}
where $(v^{\{T\}},\eta^{\{T\}})(t)\in U \times \mathbb{R}$, $t \in (0,T)$ is the solution to the constrained least squares problem :
\begin{equation}
\label{constraint}
\begin{split}
&\min \int_0^T( \|v^{\{T\}}\|^2+\alpha (\eta^{\{T\}})^2)dt \\
&\textrm{s.t.} \quad \frac{dv^{\{T\}}}{dt}=(Df(u,s))v^{\{T\}}+\partial_s f(u,s)+\eta^{\{T\}}f(u,s),
\end{split}
\end{equation}
where $\alpha$ is any positive constant and $||.||$ is the Euclidean norm in $U$.\\
\end{theorem}

Here the linearized operators are defined as :

\begin{equation}
\begin{split}
(DJ(u,s))v&:=(D_vJ)(u,s):=\lim_{\epsilon \to 0}\frac{J(u+\epsilon v,s)-J(u,s)}{\epsilon}\\
(Df(u,s))v&:=(D_vf)(u,s):=\lim_{\epsilon \to 0}\frac{f(u+\epsilon v,s)- f(u,s)}{\epsilon}\\
\partial_sJ(u,s)&:=\lim_{\epsilon \to 0}\frac{J(u,s+\epsilon)-J(u,s)}{\epsilon}\\
\partial_s f(u,s)&:=\lim_{\epsilon \to 0}\frac{f(u,s+\epsilon)-f(u,s)}{\epsilon}
\end{split}
\end{equation}
$(DJ)$,$(\partial_sJ)$,$(Df)$ and $(\partial_sf)$ are a $1\times m$ vector, a scalar, an $m\times m$ matrix and an $m \times 1$ vector, respectively, representing the partial derivatives.\\

\section{Uniform hyperbolicity}
\label{sec:uniformhyperbolicity}
In order to proceed to the presentation of the uniform hyperbolicity properties, we need first to derive from equation (\ref{difeq}) the tangent linear model:
\begin{align}
\label{TLM}
\left \{ \begin{array}{ll}
\frac{dv}{dt}=Df(u,s)v\\
v(0)=v_0 & v_0 \in U
\end{array}\right.
\end{align}
where $\{v\}_t$ is the perturbation around the reference trajectory which solves (\ref{difeq}) when the differential equation is linearized locally around this trajectory. Based on the linearity of (\ref{TLM}), we deduce that :
\begin{align}\label{eq1}
v(t)=M(u_0,t) v_0
\end{align}
where $M(u_0,t)$ is a linear operator. Intuitively, $M$ should be understood as follows: the initial perturbation $v_0$ of the reference trajectory becomes $v(t)$ after time lag $t$. We can easily derive some general properties for the operator $M$:
\begin{align}
\frac{dM}{dt}=Df\cdot M
\end{align}
and we also know that $M(u_0, 0)$ is the identity operator for any $u_0$.\\

We say that the dynamical system (\ref{difeq}) has a compact, global, uniformly hyperbolic attractor $\Lambda \subset U$ at $s$ if:\\
\begin{enumerate}
\item For all $u_0 \in U$, $\textrm{dist}(\Lambda,u(t))\xrightarrow{t\to \infty}0$ where $\textrm{dist}$ is the distance arising from the inner product in $U$.
\item There is a $C\in (0,+\infty)$ and $\lambda \in (0,1)$, such that for all $u \in \Lambda$, there is a splitting of $U$ representing the space of perturbations around $u$ :
\begin{equation}
\label{space_decomp}
U = V^+(u)\oplus V^-(u)\oplus V^0(u)
\end{equation}
where the subspaces are :
\begin{itemize}
\item $V^+(u):=\{v\in U /\quad \|M(u,t)\cdot v\|\leq C\lambda^{-t}\|v\|, \forall t <0\}$ is the \textit{unstable subspace} at $u$,
\item $V^-(u):=\{v\in U /\quad \|M(u,t)\cdot v\|\leq C\lambda^{-t}\|v\|, \forall t >0\}$ is the \textit{stable subspace} at $u$.
\item $V^0(u):=\{ \alpha f(u,s), \forall \alpha \in \mathbb{R}\}$ is the \textit{neutral subspace} at $u$.
\end{itemize}
$V^-(u)$,$V^+(u)$ and $V^0(u)$ are all continuous with respect to $u$. \\
\end{enumerate}
If $r=r^+ + r^- +r^0$ with $r^+\in V^+(u)$, $r^-\in V^-(u)$, $r^0\in V^0(u)$ and $u\in \Lambda$, the continuity of the three subspaces and the compactness of $\Lambda$ implies that:
\begin{align}
\inf_{u,r^+,r^-,r^0}\frac{\|r^++r^-+r^0\|}{\max(\|r^+\|,\|r^-\|,\|r^0\|)}=\beta>0
\end{align}
This is because if $\beta=0$, then by the continuity of $V^+(u)$, $V^-(u)$, $V^0(u)$ and the compactness of $\Lambda$, there must be a $(u,r^+,r^-,r^0)$ such that $\max(\|r^+\|,\|r^-\|,\|r^0\|)=1$ and $r^++r^-+r^0=0$ which contradicts assumption (\ref{space_decomp}). Thus:

\begin{align}
\label{beta_ineq}
\|r^+\|\leq\frac{\|r\|}{\beta} \qquad  \|r^-\|\leq\frac{\|r\|}{\beta} \qquad \|r^0\|\leq\frac{\|r\|}{\beta}
\end{align}

The \textit{stable}, \textit{unstable} and \textit{neutral subspaces} are also \textit{invariant} under $M$, which means that for all $t$ and $t'$:
 
\begin{align}
\label{invariance}
\left \{ \begin{array}{ll}
v\in V^+(u(t)) &\Leftrightarrow \quad M(u(t),t')v\in V^+(u(t+t'))\\
v\in V^-(u(t)) &\Leftrightarrow \quad M(u(t),t')v\in V^-(u(t+t'))\\
v\in V^0(u(t)) &\Leftrightarrow \quad M(u(t),t') v\in V^0(u(t+t'))\\
\end{array}\right.
\end{align}

 Because of their relative simplicity, studies of uniformly hyperbolic dynamical systems (also known as "ideal chaos") have provided a lot of insight into the properties of chaotic dynamical systems \cite{hyperbolic}. Although most real-life dynamical systems are not uniformly hyperbolic, they can be classified as \textit{quasi-hyperbolic}: results obtained on hyperbolic systems can often be generalized to them \cite{quasi}. This proof covers the convergence os LSS for uniform hyperbolic flows, nevertheless, numerical results have shown that the algorithm also works for non-ideal chaos \cite{explosion}.\\

\section{Neutral subspace and time dilation}
\label{sec4}

We introduce time dilation: "running time" becomes $\tau(t)$. The differential equation becomes: 

\begin{align}
\label{difeqdilated}
\left \{ \begin{array}{ll}
\frac{du(\tau(t))}{dt}=(1+\eta(t))f(u(\tau(t)),s)\\
\eta(t)=\frac{d\tau}{dt}-1\\
u(\tau(0))=u_0 \\
\tau(0)=0
\end{array}\right.
\end{align} 

The cost function becomes:
\begin{align}
\langle J \rangle = \lim_{t\to \infty}\frac{1}{\tau(T)}\int_{0}^T J(u(\tau(t),s),s)(1+\eta(t))dt
\end{align}

The new tangent linear model is:
\begin{align}
\left \{ \begin{array}{ll}
\frac{dv(\tau(t))}{dt}=(1+\eta(t))Df(u(\tau(t)),s)\cdot v(\tau(t))+\delta(t)f(u(\tau(t)),s)\\
\tau(0)=0\\
u(0)=u_0 & u_0 \in U
\end{array}\right.
\end{align}

where $v$ is the perturbation of $u$ and $\delta$ the perturbation of $\eta$.
Consequently, if the parameter $s$ changes infinitesimally and for a reference solution where $\eta=0$ (which means $t=\tau$), $(v,\delta)$ should satisfy the following differential equation: 
\begin{align}
\frac{dv(t)}{dt}=Df(u(t),s)\cdot v(t)+\delta(t)f(u(t),s)+\frac{\partial f}{\partial s}(t)
\end{align}

The previous equation has many solutions but we can define and show that one "canonical" solution exists ($v(t)$ has no component in the \textit{neutral subspace}).\\

\section{Structural stability and the shadowing direction}
\label{sec5}
In this section, we will prove a variant of the shadowing lemma for the purpose of defining the shadowing direction and prove its existence and uniqueness. The hyperbolic structure ensures the structural stability of the attractor $\Lambda$ under perturbation in $s$ \cite{struc}, \cite{struc2}. Without loss of generality, we will assume that $s=0$ and $\tau(t)=t$ (no time dilation).\\

\begin{theorem}[Shadowing trajectory]
\label{theorem1}
If the system is uniformly hyperbolic and $M$ continuously differentiable with respect to $s$ and $u$, then for all $\{u_0(t), t\in \mathbb{R}\}\subset\Lambda$ satisfying the differential equation (\ref{difeqdilated}) with $s=0$ and $\tau_0(t)=t$, there is a $L>0$ such that for all $|s|<L$ there is $\{(u_s(\tau_s(t)),\tau_s(t)),t\in\mathbb{R}\}$ satisfying $||u_s(\tau_s(t))-u_0(t)||<L$, $||\frac{d\tau_s(t)}{dt}||=||1+\eta_s(t)||<L$  and $\frac{du_s(\tau_s(t))}{dt}=(1+\eta_s(t))f(u_s(\tau_s(t)),s)$ for all $t\in\mathbb{R}$. Furthermore, $u_s$ and $\tau_s$ are uniformly continuously differentiable with respect to $s$.\\
\end{theorem}
The uniform continuous differentiability of $u_s$ and $\tau_s$ means that for all $s \in (-L,L)$ and $\epsilon >0$ there is a $\delta >0$ such that if $|s-s'|<\delta$ then $ \|\frac{du_s(\tau_s(t))}{ds}-\frac{du_{s'}(\tau_{s'}(t))}{ds}\|< \epsilon$ and $|\frac{d\tau^s}{ds}(t)-\frac{d\tau_{s'}}{ds}(t)|< \epsilon$ for all $t$.\\

To prepare for the proof, let $\mathbf{B}$ be the space of $\mathcal{C}^\infty$ bounded functions in $U$ and $V_t$ the hyperplane of $U$ defined by $V_t=V^+(u_0(t))\oplus V^-(u_0(t))$. We introduce $\mathbf{V}\subset \mathcal{L}^\infty$ as the space of bounded functions $\{r(t),t\in\mathbb{R}\}$ such that $r(t) \in V_t$ for all $t \in \mathbb{R}$ ($r(t)$ has no components in the \textit{neutral subspace}).
Finally, by considering the space $\mathbf{T}$ of $\mathcal{C}^\infty$ bounded functions in $\mathbb{R}$, we denote $\mathbf{A}$ the product of $\mathbf{V}$ by $\mathbf{T}$ : 
$$\mathbf{A}=\mathbf{V} \times \mathbf{T}$$
We then introduce the notation $(\textbf{r},\boldsymbol\tau)=\{(r(t),\tau(t)),t\in\mathbb{R}\}\in \mathbf{A}$ where $r(t)\in \mathbf{V}$, $\tau(t) \in\mathbf{T}$ and define the norm :
$$||(\textbf{r},\boldsymbol\tau)||_\mathbf{A}=\sup_{t\in\mathbb{R}}\big(\|r(t)\|\big)+\sup_{t\in\mathbb{R}}\big(|\tau(t)|\big)=\|\textbf{r}\|_\infty + \|\boldsymbol\tau\|_\infty$$
As defined above, the space $\mathbf{A}$ is a Banach space. 

We can now define the map $F$ : $\mathbf{A} \times \mathbb{R} \to \mathbf{B}$ as :
$$ \forall (\mathbf{r},\boldsymbol\tau)\in \mathbf{A},\forall s\in \mathbb{R}, \quad F((\mathbf{r},\boldsymbol\tau),s)=\Big\{\frac{d(u_0+r)}{dt}(\tau(t))-(1+\eta(t))f\big((u_0+r)(\tau(t)),s\big),t \in \mathbb{R} \Big\}$$
where, as seen previously, $1+\eta(t)=\frac{d\tau(t)}{dt}$.\\
For a given $s$, $F((\mathbf{r},\boldsymbol\tau),s)=\mathbf{0}$ if and only if $\{(u_0+r)(t), t\in \mathbf{R}\}$ satisfies the differential equation (\ref{difeqdilated}) where $\tau$ is the time dilation function. We use the implicit function theorem to complete the proof, which requires F to be differentiable with respect to $(\mathbf{r},\boldsymbol\tau)$ and its derivative to be non-singular at $\mathbf{r}=\mathbf{0}$, $\boldsymbol\tau=\textbf{Id}$ \footnote{$\textbf{Id}$ is the identity function: $\tau(t)=t$ for all $t\in \mathbb{R}$} and $s=0$.\\

\begin{lemma}
Under the conditions of theorem \ref{theorem1}, F has a Fréchet derivative at all $(\mathbf{r},\boldsymbol\tau)\in \mathbf{A}$ and $|s|<L$:
\begin{align}
(DF((\mathbf{r},\boldsymbol\tau),s))(\mathbf{w},\boldsymbol\epsilon)=&\Big\{\frac{dw(\tau(t))}{dt}-(1+\eta(t))Df((u_0+r)(\tau(t)),s)\cdot w(\tau(t))-\epsilon(t)f((u_0+r)(\tau(t)),s),t\in \mathbb{R}\Big\}
\end{align}
where $(\mathbf{w},\boldsymbol\epsilon) \in \mathbf{A}$.\\
\end{lemma}

The proof is quite straight forward and based on the fact that $Df$ and $f$ are uniformly continuous and bounded on the compact set $\Lambda$. \\

\begin{lemma}
Under conditions of theorem \ref{theorem1}, the Fréchet derivative of $F$ at $(\mathbf{r},\boldsymbol\tau)=(\mathbf{0},\mathbf{1})$ and $s=0$ is a bijection.
\end{lemma}
\begin{proof}
The Fréchet derivative of $F$ at $(\mathbf{r},\boldsymbol\tau)=(\mathbf{0},\mathbf{Id})$ and $s=0$ in the direction $(\mathbf{w},\boldsymbol\epsilon)$ is :
$$(DF((\mathbf{0},\mathbf{Id}),0))(\mathbf{w},\boldsymbol\epsilon)=\{\frac{dw(t)}{dt}-Df(u_0(t),0)w-\epsilon(t) f(u_0(t),s),t\in \mathbb{R}\}$$
To prove its bijectivity, we only need to show that for any $\mathbf{g}=\{g(t),t\in\mathbb{R}\} \in \mathbf{B}$ there is a unique $(\mathbf{w},\boldsymbol\epsilon)\in \mathbf{A}$ such that $(DF((\mathbf{0},\mathbf{Id}),0))(\mathbf{w},\boldsymbol\epsilon)=\mathbf{g}$\\
In this case, we can find an analytical expression for the pre-image of $\mathbf{g}$. Let $(\mathbf{w},\boldsymbol\epsilon)$ be defined as :
\begin{align}\label{shaddir}
\left \{ \begin{array}{ll}
w(t)=-\int_{-\infty}^{+\infty} M(u_0(x),t-x)\cdot\bigg(g^+(x)\mathbf{1}_{\{t<x\}}+g^-(x)\mathbf{1}_{\{t>x\}}\bigg) dx\\
\epsilon(t)=-g^0(t)\cdot \frac{f(u_0(t),0)}{||f(u_0(t),0)||^2}
\end{array}\right.
\end{align}
where $g^+(x)\in V^+(u_0(x))$, $g^-(x)\in V^-(u_0(x))$ and $g^0(x)\in V^0(u_0(x))$. We can verify that $\frac{dw}{dt}-(Df(u_0(t),s))w(t)-\epsilon(t) f(u_0(t),s)=g(t)$ for all $t$.\\
We still have to ensure that $(\mathbf{w},\boldsymbol\epsilon)$ belongs to $\mathbf{A}$. Based on (\ref{invariance}), we notice that the $w(t)$ we have just defined belongs to $V_{t}=V^+(u_0(t))\oplus V^-(u_0(t))$. Since $V^+(u_0)$, $V^-(u_0)$ and $V^0(u_0)$ are continuous with respect to $u_0$ and $\Lambda$ is compact:
\begin{align}
\max(\|g^+(t)\|,\|g^-(t)\|,\|g^0(t)\|)\leq \frac{\|g(t)\|}{\beta}\leq \frac{\|\mathbf{g}\|_{\mathbf{B}}}{\beta} \quad \textrm{for all } t
\end{align}
where $\beta >0$.\\ 
Consequently, for all $i$:
\begin{align}
\|w(t)\|&\leq \int_{t}^{+\infty} \|M(u_0(x),t-x)g^+(x)\|dx + \int_{-\infty}^{t} \|M(u_0(x),t-x)g^-(x)\|dx \\
&\leq \int_{t}^{+\infty} C\lambda^{x-t}\frac{\|\mathbf{g}\|_\infty}{\beta} dx + \int_{-\infty}^{t} C\lambda^{t-x}\frac{\|\mathbf{g}\|_\infty}{\beta} dx \\
&\leq \frac{-2C\|\mathbf{g}\|_\infty}{ln(\lambda)}
\end{align}  
thanks to the uniform hyperbolicity properties of $M$. Thus, $\mathbf{w}$ is bounded and we can conclude that $\mathbf{w}\in \mathbf{\mathrm{V}}$. 
On the other hand, we can easily show that for all $t$:
\begin{align}
\epsilon (t) &\leq \frac{\|g^0(t)\|}{\|f(u_0(t),0)\|}\\
&\leq \frac{\|\mathbf{g}\|_{\infty}}{\beta m}
\end{align}
where $m=\inf_{u\in\Lambda}\big\{\|f(u,0)\|\big\}>0$. Consequently, $\boldsymbol\epsilon$ is uniformly bounded which leads to $\boldsymbol\epsilon \in \mathrm{\textbf{T}}$ and $(\mathbf{w},\boldsymbol\epsilon) \in \mathbf{A}$.\\

Because of linearity, uniqueness of $(\mathbf{w}, \boldsymbol\epsilon)$ such that $(DF((\mathbf{0},\mathbf{Id}),0))(\mathbf{w},\boldsymbol\epsilon)=\mathbf{g}$ only needs to be proved for $\mathbf{g}=\mathbf{0}$. Since $U= V^+(u_0)\oplus V^-(u_0) \oplus V^0(u_0)$, $g(t)=0$ is equivalent to $g^+(t)=g^-(t)=g^0(t)=0$. Thanks to property (\ref{invariance}), by splitting $w(t)=w^+(t) + w^-(t)$ and knowing that $\epsilon(t)f(u_0(t),s)\in V^0(u^0)$ , we have:
\begin{align}
0=g^+(t) + g^-(t)=\big(\frac{dw^+}{dt}-(Df(u_0(t),0))w^+(t)\big)+\big(\frac{dw^-}{dt}-(Df(u_0(t),0))w^-(t)\big)
\end{align} 
where the two parentheses are in $V^+(u_0(t))$ and $V^-(u_0(t))$ respectively. Again knowing that $U= V^+(u_0)\oplus V^-(u_0) \oplus V^0(u_0)$, both parentheses should be equal to zero. This is true for all $t$, so we obtain the following two differential equations:
\begin{align}
\frac{dw^+}{dt}=(Df(u_0(t),0))w^+(t)\\
\frac{dw^-}{dt}=(Df(u_0(t),0))w^-(t)
\end{align} 
We notice that these differential equations are very similar to the linear tangent model (\ref{TLM}). Their solution is explosive unless $w(t)=0$ for all $t$. 
Showing that $\epsilon(t)=0$ is trivial:
\begin{align}
0=g^0(t)=-\epsilon(t)f(u_0(t),s)
\end{align}
Since $\|f(u_0(t),s)\| \geq m>0$\footnote{We assume that this inequality holds almost everywhere, otherwise our dynamical system would allow many degenerate trajectories.} then $\boldsymbol\epsilon=0$. This proves the uniqueness of $(\mathbf{w},\boldsymbol\epsilon)$ for $\mathbf{g}=0$.
\end{proof}\\

\begin{proof}(\textit{of theorem \ref{theorem1}})
Since $F((\mathbf{0},\mathbf{Id}),0)=\{\frac{du_0}{dt}(t) - f(u_0(t),0),t\in\mathbb{R}\}=\mathbf{0}$, $(\mathbf{0},\mathbf{Id})$ is a zero point of $F$ at $s=0$. Based on this information and on the two previous lemmas, the implicit function theorem states that there exist $L>0$ such that for all $|s|<L$ there is a unique $(\mathbf{r}^s,\boldsymbol\tau^s)$ satisfying $\|(\mathbf{r}^s,\boldsymbol\tau^s)\|_{\mathbf{A}}<L$ and $F((\mathbf{r}^s,\boldsymbol\tau^s),s)=\mathbf{0}$. 
Furthermore, this $(\mathbf{r}^s,\boldsymbol\tau^s)$ is continuously differentiable to $s$, i.e., $\frac{d(\mathbf{r}^s,\boldsymbol\tau^s)}{ds}\in \mathbf{A}$ is continuous with respect to $s$ in the $\mathbf{A}$ norm. By the definition of derivatives (in $\mathbf{A})$, $\frac{d(\mathbf{r}^s,\boldsymbol\tau^s)}{ds}=\big\{(\frac{dr^s}{ds}(t),\frac{d\tau^s}{ds}(t)),t\in\mathbb{R}\big\}$. Continuity of $\frac{d(\mathbf{r}^s,\boldsymbol\tau^s)}{ds}$ in $\mathbf{A}$ then implies that 
$\frac{dr_i^s}{ds}$ and $\frac{d\tau_i^s}{ds}$ are uniformly continuous with respect to $s$. By defining:
\begin{align}
\{(u^s(t),\tau^s(t)), t\in \mathbb{R}\}=\{(u_0(t)+r^s(t),\tau^s(t)), t\in \mathbb{R}\}
\end{align}
we finally obtain the results of theorem (\ref{theorem1}).
\end{proof}\\

This theorem states that for a trajectory $\{u_0(t),t\in \mathbb{R}\}$ satisfying (\ref{difeq}) for $s=0$, there is $\{(u^s(t),\tau^s(t),t\in \mathbb{R}\}$ satisfying the time dilated differential equation (\ref{difeqdilated}) at nearby values of $s$. In addition, $(\mathbf{u^s},\boldsymbol\tau^s)$ \textit{shadows} $(\mathbf{u_0},\mathbf{Id})$ meaning that  $(\mathbf{u^s},\boldsymbol\tau^s)$  is close to $(\mathbf{u_0},\mathbf{Id})$ when $s$ is close to 0. Also, $\frac{d(\mathbf{u}^s,\boldsymbol\tau^s)}{ds}$  exists and is uniformly bounded.\\

The \textit{shadowing direction} $(v^{\{\infty\}}(t),\eta^{\{\infty\}}(t),t\in\mathbb{R})$ is defined as the uniformly bounded series :
\begin{align}
\big\{(v^{\{\infty\}}(t),\eta^{\{\infty\}}(t))\big\}:=\bigg\{\big(\left.\frac{du^{s\{\infty\}}}{ds}\right|_{s=0}(t),\left.\frac{d^2\tau^{s\{\infty\}}}{ds dt}\right|_{s=0}(t)\big)\bigg\}\in \mathbf{A}
\end{align}
In addition, we can find 2 constants $\|\mathbf{v}^{\{\infty\}}\|$ and $\|\boldsymbol\eta^{\{\infty\}}\|$ such that for all $t$:
\begin{align}
v^{\{\infty\}}(t)\leq \|\mathbf{v}^{\{\infty\}}\| \quad \textrm{and} \quad \eta^{\{\infty\}}(t) \leq \|\boldsymbol\eta^{\{\infty\}}\|
\end{align}
We know the explicit expression of the \textit{shadowing direction}: we just need to replace $g$ by $\frac{\partial f}{\partial s}$ in (\ref{shaddir}) and the bounds found earlier are still valid ($\|\frac{\partial f}{\partial s}\|$ is bounded on the compact $\Lambda$). \\

\section{A simpler result}
\label{sec6}
In this section, we prove an easier version of Theorem LSS in which we replace the solution $\big\{(\textbf{v}^{\{T\}},\boldsymbol\eta^{\{T\}})\}$ to the constrained least squares problem (\ref{constraint}) by the shadowing direction we found earlier $\big\{(\textbf{v}^{\{\infty\}},\boldsymbol\eta^{\{\infty\}})\}$.\\ 

\begin{theorem}
If uniform hyperbolicity holds, $M$ is continuously differentiable and for all continuously differentiable function $J:\mathbf{R}^m \times \mathbf{R}\to \mathbf{R}$ whose infinite time average :
\begin{equation}
\langle J\rangle(s)=\lim_{T \rightarrow +\infty} \frac{1}{T}\int_0^T J(u(t),s) dt \quad \textrm{where} \quad \frac{du}{dt}=f(u,s) \quad \textrm{and} \quad u(0)=u_0
\end{equation}
is independent of the initial state $u_0$, let $\big\{(\textbf{v}^{\{\infty\}},\boldsymbol\eta^{\{\infty\}})\}$ be the shadowing direction, then:
\begin{align}
\frac{d \langle J \rangle}{ds}&=\lim_{T \to \infty} \frac{1}{T}\int_0^T \Big[(DJ(u,0))v^{\{\infty\}}+\partial_sJ(u,0)+\eta^{\{\infty\}}\big(J(u,0)-\langle J\rangle(0)\big)\Big]dt\\
\end{align}
\end{theorem}

\begin{proof}
The proof is essentially an exchange of limits through uniform convergence. Since $\langle J\rangle$ is independent of $u_0$, we set $u_0=u^{s}(0)$ as defined in the previous section and we know that $\frac{du^{s}(\tau^{s}(t))}{dt}=(1+\eta^s)f(u^{s}(t),s)$. We can write:
\begin{align*}
&\left.\frac{d \langle J \rangle}{ds}\right|_{s=0}=\lim_{s\rightarrow 0}\frac{\langle J \rangle(s)-\langle J \rangle (0)}{s}\\
&=\lim_{s\rightarrow 0}\lim_{T \rightarrow +\infty}\bigg( \frac{1}{\tau^s(T) \times s}\int_{0}^T J(u^s(\tau^s(t)),s)(1+\eta^s(t))dt-\frac{1}{T \times s}\int_0^TJ(u^0(t),0)dt\bigg)\\
&=\lim_{s\rightarrow 0}\lim_{T \rightarrow +\infty}\bigg( \frac{1}{s}\int_{0}^T \frac{J(u^s(\tau^s(t)),s)(1+\eta^s(t))}{\tau^s(T)}+\frac{J(u^0(t),0)}{\tau^s(T)}-\frac{J(u^0(t),0)}{\tau^s(T)}-\frac{J^0(u(t),0)}{T}dt\bigg)\\
&=\lim_{s\rightarrow 0}\lim_{T \rightarrow +\infty}\bigg( \frac{1}{s}\int_{0}^T \frac{J(u^s(\tau^s(t)),s)-J^0(u(t,0))+\eta^s(t)J^s(u(\tau^s(t)),s)}{\tau^s(T)}+\frac{TJ(u^0(t),0)-\tau^s(T)J(u^0(t),0)}{T\tau^s(T)}dt\bigg)\\
&=\lim_{s\rightarrow 0}\lim_{T \rightarrow +\infty}\bigg( \frac{1}{s}\int_{0}^T \frac{J(u^s(\tau^s(t)),s)-J(u^0(t),0)+\eta^s(t)J(u^s(\tau^s(t)),s)}{\tau^s(T)}+\frac{TJ(u^0(t),0)-\int_0^T(1+\eta^s(t'))dt'J(u^0(t),0)}{T\tau^s(T)}dt\bigg)\\
&=\lim_{s\rightarrow 0}\lim_{T \rightarrow +\infty}\bigg( \frac{1}{s}\int_{0}^T \frac{J(u^s(\tau^s(t)),s)-J(u^0(t),0)+\eta^s(t)J(u^s(\tau^s(t)),s)}{\tau^s(T)}-\frac{\int_0^T(\eta^s(t'))dt'J(u^0(t),0)}{T\tau^s(T)}dt\bigg)\\
&=\lim_{s\rightarrow 0} \lim_{T \rightarrow +\infty} \frac{1}{\tau^s(T)}\int_{0}^T \frac{J(u^s(\tau^s(t)),s)-J(u^0(t),0)}{s}+\frac{\eta^s(t)\Big(J(u^s(\tau^s(t),s)-\frac{\int_0^T J(u^0(x),0)dx}{T}\Big)}{s}dt
\end{align*}

Let us eliminate $\lim_{s\rightarrow 0}$ in the first term. We define :
\begin{equation}
\gamma^s(t)=\frac{dJ(u^s,s)}{ds}(t)=(DJ(u^s(t),s))\frac{du^s(t)}{ds}+\partial_sJ(u^s(t),s)
\end{equation}
Then, thanks to the mean value theorem, for all $t$ there exist an $\xi_t(s) \in[0,s]$ such that:
\begin{equation}
\frac{J(u^s(t),s)-J(u^0(t),0)}{s}=\gamma^{\xi_t(s)}(t)
\end{equation}
Consequently:
\begin{equation}
\lim_{s\rightarrow 0}\lim_{T \rightarrow + \infty} \bigg( \frac{1}{\tau(T)}\int_0^T\frac{(J(u^s(t),s)-J(u^0(t),0))}{s}dt\bigg)=\lim_{s\rightarrow 0}\lim_{T \rightarrow + \infty} \bigg( \frac{1}{\tau(T)}\int_0^T\gamma^{\xi_t(s)}(t)dt\bigg)\\
\end{equation}
We can choose a neighborhood of $\Lambda \times \{0\}$ that contains $(u^s(t),	s)$ for all $t$ (for $s$ sufficiently small) and in which both $(DJ(u,s))$ and $\partial_sJ(u,s)$ are uniformly continuous. Since the $\frac{du^s}{ds}(t)$ are uniformly continuous and bounded, for all $\epsilon>0$ there exists $L>0$ such that for all $|\xi|<L$:
$$\|\gamma^{\xi}(t)-\gamma^0(t)\|<\epsilon \quad \forall t$$
Thus, for all $|s|<L$, $|\xi(s)|\leq|s|<L$ for all $t$, therefore for all $T$ :
\begin{equation}
\Bigg\|\frac{1}{\tau^s(T)}\int_0^T\gamma^{\xi_t(s)}(t)-\frac{1}{\tau^s(T)}\int_0^T\gamma^{0}(t)dt\Bigg\|\leq \frac{1}{\tau(T)}\int_0^T\|\gamma^{\xi_t(s)}(t)-\gamma^0(t)dt\|\leq \frac{T}{\tau^s(T)}\epsilon \leq (1+\|\boldsymbol\eta\|_\infty)\epsilon
\end{equation}
Hence,
\begin{equation}
\Bigg\| \lim_{T \rightarrow + \infty}\bigg( \frac{1}{\tau^s(T)}\int_0^T \gamma^{\xi_t(s)}(t)\bigg)-\lim_{T \rightarrow + \infty} \bigg( \frac{1}{\tau^s(T)}\int_0^T\gamma^{0}(t)\bigg)\Bigg\| \leq (1+\|\boldsymbol\eta\|_\infty)\epsilon
\end{equation}
Finally,
\begin{equation}
\lim_{s\rightarrow0} \lim_{T \rightarrow + \infty} \bigg( \frac{1}{\tau^s(T)}\int_0^T\gamma^{\xi_t(s)}(t)dt\bigg)= \lim_{T \rightarrow + \infty}\bigg( \frac{1}{T}\int_0^T\gamma^{0}(t)dt\bigg)
\end{equation}
which grants us the desired result for the first term via the definition of $\gamma_i^0$.\\

For the second term, $J$ is continuously differentiable thus continuous and the $(u_i^s,\tau_i^s)$ are $i$-uniformly continuously differentiable and bounded. Based on that, for $s$ sufficiently small, we can find a compact neighborhood of $\Lambda \times \{0\}$ that contains $(u_i^s,s)$ for all $i\in \mathbf{Z}$ and in which $J(u,s)$ will be uniformly continuous. Consequently, $\big\{\frac{\eta^s}{s}J(u^s(\tau^s(t)),s),t\in\mathbf{R}^+\big\}$ which can be written $\big\{\frac{\frac{d\tau^s(t)}{dt}-\frac{d\tau^0(t)}{dt}}{s}J(u^s(\tau^s(t)),s),t\in\mathbf{R}^+\big\}$  converges uniformly to $\big\{\left.\frac{d\tau^{s\{\infty\}}}{ds dt}\right|_{s=0}(t)J(u^s(\tau^s(t)),s),t\in\mathbf{R}^+\big\}$ when $s$ goes to $0$. Because the term $\frac{\int_0^T J(u^0(x),0)dx}{T}$ does not depend on $s$ at all, we finally have:
$$\lim_{s\rightarrow 0}\lim_{T \rightarrow + \infty}\frac{1}{\tau^s(T)}\int_0^T \frac{\eta^s(t)\Big(J(u^s(\tau^s(t),s)-\frac{\int_0^T J(u^0(x),0)dx}{T}\Big)}{s}=\lim_{T \rightarrow + \infty}\frac{1}{T}\int_0^T\eta^{\{\infty\}}(t)\big(J(u^0(t),0)-\langle J\rangle(0)\big)\Big]dt$$
which conludes the proof.

\end{proof}

\section{Computational approximation of the shadowing direction}
\label{sec7}
The main task of this section is to provide a bound for  :
\begin{align}
&e^{\{T\}}(t)=v^{\{T\}}(t)-v^{\{\infty\}}(t)\\
\end{align}
for $t\in (0,T)$. $(\textbf{v}^{\{T\}},\boldsymbol\eta^{\{T\}}$) is the solution to the least squares problem:
\begin{align}
\label{least squares}
&\min \int_0^T\big( \|v^{\{T\}}(t)\|^2+\alpha (\eta^{\{T\}}(t))^2\big) dt \quad\\
\label{least squares 2}
&\textrm{s.t.} \quad \frac{dv^{\{T\}}(t)}{dt}=Df(u(t),s) v^{\{T\}}(t)+\eta^{\{T\}}(t)f(u(t),s)+\frac{\partial f}{\partial s}(t),
\end{align}
The shadowing lemma guarantees the existence of a shadowing trajectory, but provides no clear way to compute $\{(v^{\{\infty\}}(t)$, $\eta^{\{\infty\}}(t))\}$. This section suggests that the solution to the least squares problem gives a useful approximation of the shadowing trajectory allowing us to compute$\frac{d\langle J \rangle}{ds}$. Without loss of generality, we consider that $s=0$ in (\ref{least squares 2}). By definition, the shadowing trajectory satisfies: 
\begin{align}
\frac{du^{s}(\tau^{s}(t))}{dt}=(1+\eta^{s}(t))f(u^{s}(\tau^{s}(t)),s)
\end{align}
After taking the derivative to $s$ on both sides for $s=0$, we obtain:
\begin{align}
\label{infinity}
\frac{dv^{\{\infty\}}(t)}{dt}=Df(u(t),s) v^{\{\infty\}}(t)+\eta^{\{\infty\}}(t)f(u(t),s)+\frac{\partial f}{\partial s}(t)
\end{align}
Thus, the shadowing direction satisfies the constraint (\ref{least squares 2}) and:

\begin{align}
\label{inequality}
\int_0^T\big( \|v^{\{T\}}(t)\|^2+\alpha (\eta^{\{T\}}(t))^2\big) dt \leq \int_0^T\big( \|v^{\{\infty\}}(t)\|^2+\alpha (\eta^{\{\infty\}}(t))^2\big) dt \leq T(||\mathbf{v}^{\{\infty\}}||^2 + \alpha||\boldsymbol\eta^{\{\infty\}}||^2)
\end{align}

Combining the constraint equation (\ref{least squares 2}) as well as (\ref{infinity}) we obtain :
\begin{displaymath}
\left\{\begin{array}{l}
\frac{de^{\{T\}+}(t)}{dt}=Df(u(t),s)e^{\{T\}+}(t)\\
\frac{de^{\{T\}-}(t)}{dt}=Df(u(t),s)e^{\{T\}-}(t)\\
\end{array}\right.
\end{displaymath}
Consequently:
\begin{displaymath}
\left\{\begin{array}{l}
e^{\{T\}+}(t)=M(u(0),T-t)e^{\{T\}+}(T)\\
e^{\{T\}-}(t)=M(u(0),t)e^{\{T\}-}(0)\\
\end{array}\right.
\end{displaymath}
where $M$ is the operator we defined in the first section. Since $M(u,0)=\textbf{Id}$ and knowing that $\partial_tM(u,.)$ is continuous, we can find a positive constant $K$ such that for $t$ sufficiently small:
\begin{align}
\|\partial_t M(u,t)\|\leq K
\end{align} 
and for a $t$ such that $tK<1$:
\begin{align}
\|M(u(0),t)v\|\geq (1-tK)\|v\|
\end{align}
for all $v\in U$. Knowing that $M(u(0),t+t')v=M(u(t),t')(M(u(0),t)v)$, for $tK$ and $t'K$ less than $1$, we deduce that:
\begin{align}
\|M(u(0),t+t')v\|\geq (1-t'K)(1-tK)\|v\|
\end{align}
We can iterate this process and refine the timesteps to obtain:
\begin{align}
\|M(u(0),t)v\|\geq e^{-Kt}\|v\|
\end{align}
for any $t$ this time.\\
Consequently:
\begin{align}
\int_0^T \|e^{\{T\}-}(t)\|^2dt \geq \int_0^T e^{-Kt}\|e^{\{T\}-}(0)\|^2dt \\
\geq \|e^{\{T\}-}(0)\|^2 \times \frac{1-e^{-KT}}{K}\label{a1}
\end{align}

On the other hand, since $e^{\{T\}-}(t)=v^{\{T\}-}(t)-v^{\{\infty\}-}(t)$, then:
\begin{align}\label{a2}
\|e^{\{T\}-}(t)\|^2\leq 2\big(\|v^{\{T\}-}(t)\|^2+\|v^{\{\infty\}-}(t)\|^2\big)\leq \frac{2}{\gamma}\big(\|v^{\{T\}}(t)\|^2+\|v^{\{\infty\}}(t)\|^2\big)
\end{align}

Combining (\ref{a1}), (\ref{a2}) and (\ref{inequality}), we obtain:
\begin{align}
\|e^{\{T\}-}(0)\|^2&\leq \frac{K}{1-e^{-KT}}\int_0^T\frac{2}{\gamma}\big(\|v^{\{T\}}(t)\|^2+\|v^{\{\infty\}}(t)\|^2\big)dt\\
&\leq \frac{2K}{\gamma(1-e^{-KT})}\big(\|\eta^{\{\infty\}}(t)\|^2+2\|v^{\{\infty\}}(t)\|^2\big)\times T
\end{align}

For $T$ sufficiently large, this means that we can find a constant $E$ such that: 

\begin{align}
\sup_{t\in (0,T)} \|e^{\{T\}-}(t)\|^2 \leq E\sqrt{T}
\end{align}
because $\|e^{\{T\}-}(t)\|\leq C \lambda^t\|e^{\{T\}-}(0)\|$ with $0<\lambda<1$ (uniform hyperbolicity).
In the same way we obtain :
\begin{align}
\max_i ||e_i^{\{h,T\}+}||\leq E\sqrt{T}
\end{align}

\section{Convergence of least squares shadowing} 
\label{sec8}
In this section, we use the results obtained previously to prove our initial theorem :
\begin{theorem}[THEOREM LSS] For a sufficiently smooth uniformly hyperbolic dynamical system and a $C^1$ cost function $J$, the following limit exists and is equal to:
\begin{align*}
\frac{d \langle J \rangle}{ds}&=\lim_{T \to \infty} \frac{1}{T}\int_0^T \Big[(DJ(u,0))v^{\{T\}}+\partial_sJ(u,0)+\eta^{\{T\}}\big(J(u,0)-\langle J\rangle(0)\big)\Big]dt\\
\end{align*}
\end{theorem}

\begin{proof}
Because $J$ is $C^1$ and $\Lambda$ is compact, there exists a constant $A$ such that $\|DJ(u(t),0)\|<A$ for all $t$. Let $\textbf{e}^{\{T\}}$ be defined as in the previous section, then:
\begin{align}
\Bigg|&  \frac{1}{T}\int_0^T \Big[(DJ(u,0))v^{\{T\}}+\partial_sJ(u,0)+\eta^{\{T\}}\big(J(u,0)-\langle J\rangle(0)\big)\Big]dt\\
&- \frac{1}{T}\int_0^T \Big[(DJ(u,0))v^{\{\infty\}}+\partial_sJ(u,0)+\eta^{\{\infty\}}\big(J(u,0)-\langle J\rangle(0)\big)\Big]dt\Bigg|\\
&=\Bigg| \frac{1}{T}\int_0^T \Big[(DJ(u,0))e^{\{T\}}+\epsilon^{\{T\}}\big(J(u,0)-\langle J\rangle(0)\big)\Big]dt\Bigg|\\
&=\Bigg| \frac{1}{T}\int_0^T \Big[(DJ(u,0))(e^{\{T\}+}+e^{\{T\}-}+e^{\{T\}0})+\epsilon^{\{T\}}\big(J(u,0)-\langle J\rangle(0)\big)\Big]dt\Bigg|\\
&<\Bigg| \frac{1}{T}\int_0^T \Big[(DJ(u,0))(e^{\{T\}+}+e^{\{T\}-})dt\Bigg|+\Bigg| \frac{1}{T}\int_0^T \Big[(DJ(u,0))e^{\{T\}0}+\epsilon^{\{T\}}\big(J(u,0)-\langle J\rangle(0)\big)dt\Bigg|
\end{align}

For the first term:

\begin{align}
\Bigg| \frac{1}{T}\int_0^T \Big[(DJ(u,0))(e^{\{T\}+}+e^{\{T\}-})dt\Bigg|<& \frac{1}{T}\int_0^T \|DJ(u,0)e^{\{T\}+}\|dt \\
&+\frac{1}{T}\int_0^T \|DJ(u,0)e^{\{T\}+}\|dt 
\end{align}
\begin{align}\label{sqrtT}
&\leq \frac{1}{T}\Big(\int_0^{T}C\lambda^{T-t}\|e^{\{T\}+}(T)\|dt+\int_0^{T}C\lambda^{t}\|e^{\{T\}-}(0)\|dt\Big)\\
&\leq \frac{1}{T} \frac{ 2C(1-\lambda^T)}{-\ln(\lambda)}\times E\sqrt{T} \\
&\leq \frac{1}{\sqrt{T}}\times\frac{2CE}{-\ln(\lambda)}
\end{align}

which goes to $0$ when $T$ increases. Thus, we notice that, the differences $e^{\{T\}+}$ and $e^{\{T\}-}$ between the $v^{\{\infty\}+}$ and $v^{\{\infty\}-}$ components of the \textit{shadowing direction} and their approximations $v^{\{T\}+}$ and $v^{\{T\}-}$ decrease extremely fast so that the whole term $\big| \frac{h}{T}\sum_{i=1}^{[\frac{T}{h}]}\Big[(DJ(u_i,s))(e_i^{\{h,T\}+}+e_i^{\{h,T\}-})\big|$ tends to $0$ as $O(\frac{1}{\sqrt{T}})$.\\

On the other hand, there is no reason for $e^{\{T\}0}(t)$ and $\epsilon^{\{T\}}(t)$ to decrease when $T$ increases. The cancellation of the second term is the result of the mutual cancellation of the elements in the summation as we shall see. 
Based on the shadowing trajectory $\{(u^{s}(t),\tau^{s}(t),t\in \mathbb{R}^+)\}$ found in section \ref{sec5}, we consider the new trajectory and time dilation $\{(u^{'s}(t),\tau^{s}+s\int_0^t\epsilon^{\{T\}})(t),t\in\mathbb{R+}\}$ which satisfy the following relation :
\begin{align}
\lim_{s\to 0} \frac{u^{'s}(\tau^s+s\int_0^t\epsilon^{\{T\}})-u^{s}(\tau^s(t))}{s}=e^{\{T\}0}(t)
\end{align}
for all $t$. We can notice that the new trajectory describes exactly the same continuous trajectory as the old one (we have just made a change in the time variable). We obtain by following the same operations we did in section \ref{sec6} (but upside down this time) :
\begin{align}
&\lim_{T\rightarrow \infty}\frac{1}{T}\int_0^T \Big[(DJ(u,0))e^{\{T\}0}+\int_0^t\epsilon^{\{T\}}\big(J(u,0)-\langle J\rangle(0)\big)dt\\
&=\lim_{T \rightarrow +\infty}\lim_{s\rightarrow 0} \frac{1}{\tau^s(T)+s\int_0^t\epsilon^{\{T\}}(T)}\int_{0}^T \frac{J(u^{'s}(\tau^s+s\int_0^t\epsilon^{\{T\}}),s)-J(u^s(\tau^s),s)}{s}\\
&+\frac{s\epsilon^{\{T\}}(t)\Big(J(u^{'s}(\tau^s+s\epsilon^{\{T\}},s)-\frac{\int_0^T J(u^s(x),s)dx}{T}\Big)}{s}dt\\
&=\lim_{T \rightarrow +\infty}\lim_{s\rightarrow 0}\bigg( \frac{1}{\big(\tau^s(T)+s\int_0^t\epsilon^{\{T\}}(T)\big) \times s}\int_{0}^T J(u^{'s}(\tau^s+s\int_0^t\epsilon^{\{T\}}),s)(1+\eta^s(t)+s\epsilon^{\{T\}})dt\\
&-\frac{1}{\tau^s(T) \times s}\int_0^TJ(u^s(\tau^s(t)),s)(1+\eta^s(t))dt\bigg)\\
&=0
\end{align}
This happens because both integrals are the same up to a change of time variable. This concludes the proof. 
\end{proof}\\ 
The fact of approximating an ergodic mean by an average over a finite trajectory is also a source of error in our method. If the dynamical system is mixing, the central limit theorem implies that this error decreases as $O(\sqrt{T})$.

\section{Practicable algorithm}
\label{sec9}
Based on theorem LSS we can derive the following algorithm\footnote{An adjoint version of it can be found in \cite{explosion}.} :
\begin{enumerate}
\item Fix a timestep $h$ and compute a discrete reference trajectory $u_0$, $u_1$, $u_2$,..., $u_n$\footnote{We discard the first points $u_{-n_0}$,...,$u_{-1}$ for $n_0$ sufficiently large so that we are sure to be on the attractor.}. In what follows, we use a standard RK4 scheme to obtain this trajectory. 
\item Compute $\{v_i,\eta_i\}$ by discretizing and solving the KKT set of equations :
\begin{align}
\label{algorithm}
\left \{ \begin{array}{l}
\frac{dv}{dt}-(Df) v-\partial_s f-\eta f=0\\
\frac{dw}{dt}+(Df)^Tw-v=0\\
w(0)=w(T)=0\\
\alpha \eta - w^Tf=0\\
\end{array}\right.
\end{align} 
where $w$ is the Lagrange multiplier function. For the detailed derivation of the KKT equations from the least squares formulation and how to solve it efficiently, the reader can consult \cite{explosion}. This system is well conditioned as shown in \cite{condition}. In this example, we discretized the system as following :
\begin{align}
\label{algorithm}
\left \{ \begin{array}{l}
\frac{v_{i+1}-v_i}{h}-\frac{1}{2}\big((Df(u_i,s))v_i+(Df(u_{i+1},s))v_{i+1}\big)\\
-\frac{1}{2}(\partial_sf(u_i,s)+\partial_sf(u_{i+1},s))-\eta_i\frac{u_{i+1}-u_i}{h}=0\\
\frac{w_{i+1}-w_i}{h}+\frac{1}{2}\big((Df(u_i,s))^Tw_i+(Df(u_{i+1},s))^Tw_{i+1}\big)-v_i=0\\
\alpha\eta_i-w_i^T\frac{u_{i+1}-u_i}{h}=0\\
w_0=w_n=0\\
\end{array}\right.
\end{align} 
\item Finally, compute the desired derivative :
\begin{align}
\frac{d\langle J\rangle}{ds}\approx \frac{1}{n+1}\sum_{i=0}^n\Big((DJ(u_i,s))v_i+\partial_s J(u_i,s)+\eta_i\big(J(u_i,s)- \langle J\rangle\big)\Big)
\end{align}
\end{enumerate} 	
We apply this algorithm to the 3-dimensional Lorenz 63 dynamical system introduced by Edward Lorenz to model the atmospheric convection : 
 $$\left \{ \begin{array}{ll}
\frac{dx}{dt}=\sigma(y-x)\\
\frac{dy}{dt}=x(\rho-(z-z_0))-y\\
\frac{dz}{dt}=xy-\beta (z-z_0)\\
\end{array}\right. $$
It is an autonomous ODE parameterized by $\sigma$, $\beta$, $\rho$, $z_0$ and the quantity of interest is $\langle J \rangle =\lim_{T \to \infty}\frac{1}{T}\int_{t=0}^{T}z(t) dt$, the time average of the component $z$. While fixing $\sigma=10$, $\beta=\frac{8}{3}$, $\rho=25$ and $z_0=0$, we will compute $\frac{d\langle J \rangle}{dz_0}$ which is clearly equal to $1$ (when $z_0$ increases the attractor translates in the $z$ direction).We also have an analytical expression for the \textit{shadowing direction} : $\{v_i^{\infty}=(0,0,1),\eta_i^{\infty}=0\}$ for all $i$. We set $h=0.02$ and compute $\frac{d\langle J \rangle}{dz_0}$  for different integration lengths $T$. We notice that the algorithm gives a very good estimate of the sensitivity and that this estimate improves as $T$ increases (figure \ref{figure1}). As expected, the error decreases as $O(\sqrt{T})$. Then, we fix $T=100$ and compare the computed \textit{shadowing direction} with the theoretical one for two different values of $\alpha$ (figure \ref{figure2}). First, both computations give a good approximation of $\frac{d\langle J \rangle}{dz_0}$ : $0.99$ for $\alpha=10^{16}$ and $0.96$ for $\alpha=100$. For $\alpha=10^{16}$, as we approach the "middle" of the integration length, the difference between the theoretical and the approximated \textit{shadowing direction} decreases and reaches machine precision. This comes from the expanding/contracting properties of the stable and unstable subspaces presented in section \ref{sec:uniformhyperbolicity}. As for a lower penalty $\alpha=100$ which allows a higher value for the time dilation factors, the stable and unstable components of the approximated \textit{shadowing direction} are also very close to the theoretical ones (otherwise $\log(\|v_{\textrm{approx}}-v_{\infty}\|)$ would grow exponentially) but the neutral component can be significantly different from $\eta_i^{\infty}=0$. In fact, this bigger gap is compensated by the high-valued time dilation factors $\eta_i$. Either way, both values of $\alpha$ give an acceptable estimation of the sensitivity.

\begin{figure}[htb]
\begin{center}
\includegraphics[scale=0.6]{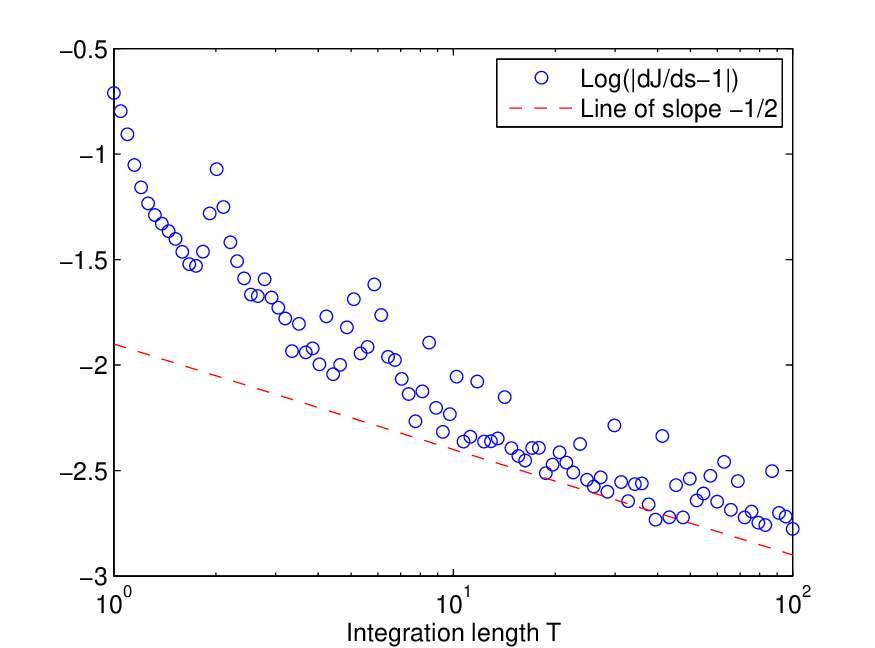}
\caption{$\log(|\frac{d\langle z \rangle}{dz_0}-1|)$ for $h=0.02$, $\alpha=100$ and different integration time lengths.}\label{figure1}
\end{center}
\end{figure}

\begin{figure}\label{figure2}
\setlength{\overfullrule}{0pt}
\begin{center}
\begin{tabular}{cc}
      \includegraphics[scale=0.48]{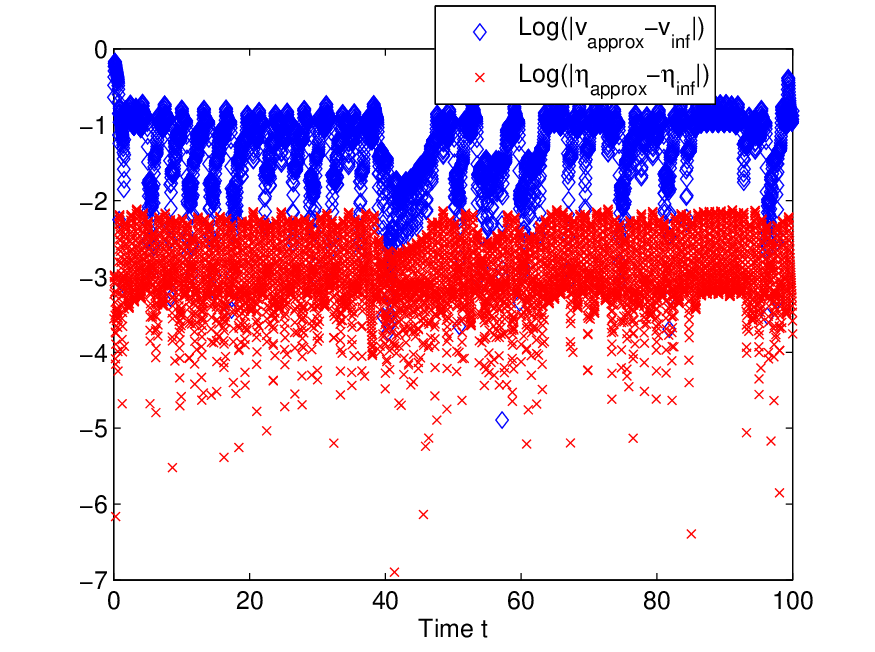} & \includegraphics[scale=0.48]{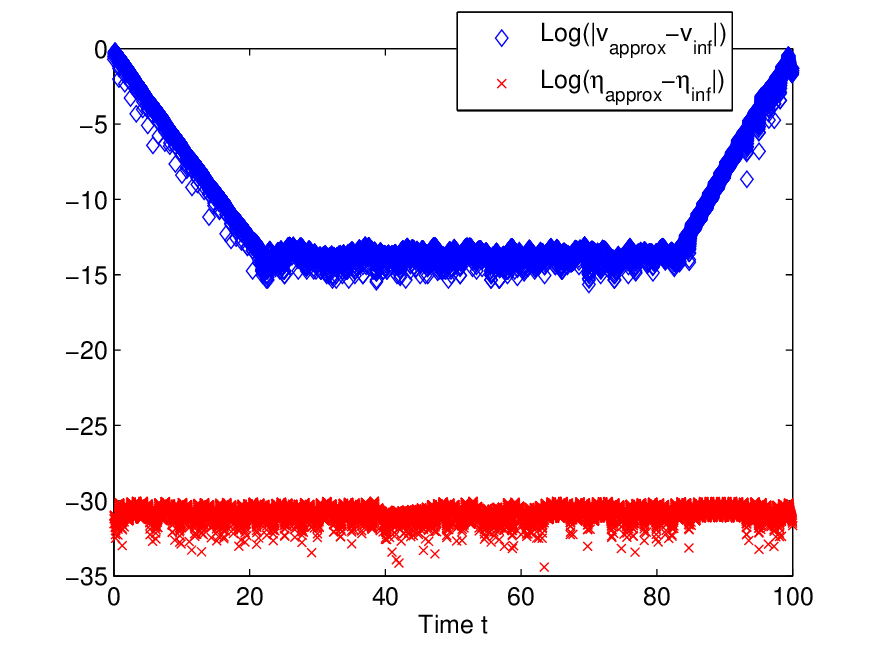} \\
      {$\alpha=100$} & {$\alpha=10^{16}$} \\
\end{tabular} 
\end{center}\caption{$\log(\|v_{\textrm{approx}}-v_{\infty}\|)$ in blue and $\log(|\eta_{\textrm{approx}}-\eta_{\infty}|)$ in red for two different values of $\alpha$}\label{figure2}
\end{figure}

\section{Conclusion}
As we have shown through this paper, LSS gives us a good estimation for $\frac{d\langle J \rangle}{ds}$ when the dynamical system is uniformly hyperbolic. After running a simulation for a given $s$ and an arbitrary initial condition $u_0$, we obtain a reference trajectory $\big\{u^s(t),t\in (0,T)\big\}$. If we had access to the \textit{shadowing direction}, we would easily compute :
\begin{align}
\label{concl}
\frac{d \langle J \rangle}{ds} \approx \frac{1}{T}\int_0^T\Big[(DJ(u^s,s))v^{\{\infty\}}+\partial_sJ(u^s,s)+\eta^{\{\infty\}}\big(J(u^s,s)-\langle J\rangle(s)\big)\Big]dt
\end{align}
However, in real-life problems we usually do not have access to the \textit{stable} and \textit{unstable subspaces} around each $u^s(t)$ prohibiting the usage of the closed form expression of $v^{\{\infty\}}$ and $\eta^{\{\infty\}}$. Thus, we have no other choice than computing an approximation of the \textit{shadowing direction}. This approximation is given by the solution to the least squares problem:
\begin{equation}
\begin{split}
&\min \int_{0}^T( \|v^{\{T\}}\|^2+\alpha (\eta^{\{T\}})^2)dt \\
&\textrm{s.t.} \quad \frac{dv^{\{T\}}}{dt}=(Df(u,s))v^{\{T\}}+\partial_s f(u,s)+\eta^{\{T\}}f(u,s),
\end{split}
\end{equation}
After solving this quadratic optimization problem, we estimate $\frac{d\langle J \rangle}{ds}$ using expression (\ref{concl}) again where the $(v^{\{\infty\}},\eta^{\{\infty\}})$ are replaced by $(v^{\{T\}},\eta^{\{T\}})$. As we have seen previously, this estimation converges to the real value of $\frac{d\langle J \rangle}{ds}$ when the integration lapse $T$ increases.

\end{document}